\documentclass[12pt]{article}

\usepackage{amsfonts}
\usepackage{amsthm}
\usepackage{amsmath}
\usepackage{amssymb}

\usepackage{appendix}

\usepackage{comment}
\usepackage{xcolor}

\usepackage{marginnote}

\newcommand{\dR}{\mathbb{R}}

\newcommand\Hd{{\mathbb{H}^d}}

\newcommand{\bE}{\mathbf{E}}
\newcommand{\bP}{\mathbf{P}}

\newcommand{\FHd}{\mathcal{F}(\Hd)}

\newcommand{\lL}{\mathcal{L}}

\newcommand{\lV}{\mathcal{V}}

\newcommand{\hOmega}{{\widehat{\Omega}}}

\newcommand{\dd}{\mathrm{d}}

\DeclareMathOperator{\spec}{Spec}
\DeclareMathOperator{\vol}{vol}

\theoremstyle{plain}
\newtheorem{theorem}{Theorem}

\newtheorem{lemma}{Lemma}[section]

\theoremstyle{definition}
\newtheorem{definition}[lemma]{Definition}

\newtheorem{example}[lemma]{Example}

\newtheorem{question}[lemma]{Question}

\theoremstyle{remark}
\newtheorem{remark}[lemma]{Remark}

\title{Quenched path limits and periodization stability for tilted Brownian motion in Poissonian potentials on $\mathbb{H}^d$}
\author{Miklos Abert, Adam Arras, Jaelin Kim}

\begin{document}

\maketitle

\begin{abstract}
We analyze the existence of Brownian motion tilted by a potential of full support on hyperbolic spaces $\Hd$. On compact spaces, it is classical that these path limits, called Q-processes, exist and can be directly defined using the ground state of the corresponding Schr\"odinger operator. On non-compact spaces like $\Hd$, the existence fails in general. 

We show that for \emph{stationary random} potentials on $\Hd$ with suitable spectral and sup norm bounds, the Q-processes exist a.s. For potentials that are factors of a Poisson point process, the method works up to sup norm $(d-1)^2/8$. In this case, we also show that the path limit can be approximated by periodic potentials.

As a tool, we use the foliated space defined by the point process. It turns out that the global ground state of this foliated space serves as a substitute for the non-existing $L^2$ ground states on the leaves of the foliation. Restricting the global ground state to a leaf gives a generalized eigenwave that can be plugged into the usual machinery to get the Q-process. 

\end{abstract}


\section{Introduction}

Let $\bP^x$ denote the Wiener measure of Brownian motion $(X_t)_{t \in \dR_{\geq 0}}$ on a smooth Riemannian manifold $M$, starting at $x \in M$. Given a continuous potential $V \colon M \to \dR_{\geq 0}$, we consider the \emph{$V$-tilted path measure of time $T$} by setting 
\begin{equation}\label{eq:tilted}
  \frac{\dd \mathbf{Q}_T^x}{\dd \bP^x}
  := \frac{1}{Z_T^x}
  \exp\!\left(- \int_0^T V(X_s)\,\dd s\right),
\end{equation}
where $Z_T^x$ is the normalizing factor. 
This measure describes the trajectory of a particle in a soft trap model, conditioned to survival until time $T$. We are interested in the existence and asymptotic behavior of the path limit (that is, the weak limit on the space of trajectories) as $T \to \infty$. When this limit exists, it is called the \emph{Q-process with respect to $V$}. 

When $M$ is compact, it is classical that the above family of path measures admits a unique weak limit as $T \to \infty$. The limit is a diffusion process whose generator is the Doob transform associated with the (unique) ground state of the Schrödinger operator $-\tfrac12 \Delta_M + V$. 

In the non-compact case, the above operator typically does not admit a lowest eigenvalue, and there is no general guarantee for the existence of the Q-process. Our main result provides a general existence criterion for homogeneous random potentials on hyperbolic spaces $M = \mathbb{H}^d$.

\begin{theorem}\label{thm:thm1}
Let $\omega$ be a homogeneous Poisson Point Process (PPP) on $\Hd$ and let
$V_\omega:\Hd\to[0,\infty)$ be a continuous factor-of-$\omega$ potential with maximum norm 
\begin{equation}\label{eq:MainThmEq}
  \|V_\omega\|_\infty < \frac{(d-1)^2}{8}.  
\end{equation}
Then for a.e.\ $\omega$ and for all $x \in \Hd$, the tilted measures $\mathbf Q_T^x$ converge to a diffusion process that is independent of the starting point $x$.
\end{theorem}

We refer to Definition \ref{def:FactorOfPPP} for a precise formulation of a factor of PPP potential. A natural family of potentials that can be used above is 
$$V_\omega(x)=\min\left\{V_{\max}, \sum_{y\in \omega} \eta(d(x,y))\right\}$$
where $\eta \colon \dR_{\geq 0} \to \dR_{\geq 0}$ is a fixed seed function and $V_{\max}$ is a constant.  Note that the intensity of the Poisson point process is irrelevant here. 

Theorem \ref{thm:thm1} brings together two well-studied objects in a novel way. Previous results on Q-processes concentrate on confining potentials, or alternatively, periodic potentials where the existence of the path limit follows from the classical theory. On the other side, Brownian motion of time $T$ tilted by Poissonian (soft or hard) traps has been extensively studied in the literature, see e.g. \cite{sznitman2013brownian} and references therein, but we did not find a discussion on the existence of the path limit. These results also concentrate on Euclidean spaces, where our spectral method stops working automatically and one needs new ideas. 

We use the spectral theory of the foliated space defined by the point process to get Theorem \ref{thm:thm1}. Indeed, in situations like ours, individual instances of $(\Hd,V_\omega)$ do not admit $L^2$ ground states. However, it turns out that the global spectral theory of $L^2$ of the ambient foliated space does admit a unique ground state that (when restricted back to leaves) can be used as a substitute in finding the limiting Q-process. 

As we will see, Theorem \ref{thm:thm1} will also hold for an arbitrary stationary random potential $V$, but the bound on the maxnorm will depend on the global spectral gap of the potential. In particular, if the underlying group action does not have spectral gap, we get an empty statement (this is exactly the issue with Euclidean spaces as noted above). In Theorem \ref{thm:thm1} we only use Poisson point processes because the spectral gap of a factor of Poisson potential is maximal possible, that is, Poisson point processes are Ramanujan, see Theorem \ref{thm:PPPisRamanujan}. 

Our next result addresses how the Q-processes defined by a general stationary random potential can be approximated by Q-processes defined by \emph{periodic potentials}. 

\begin{theorem}\label{thm:thm2}
Let $(M_n)$ be a sequence of compact hyperbolic $d$-manifolds, with uniform spectral gap 
\[
\lambda := \liminf_{n\to\infty} \lambda_1\!\left(-\tfrac{1}{2}\Delta_{M_n}\right) > 0 
\]
such that $M_n$ Benjamini-Schramm converges to $\Hd$. Let $V_n$ be a non-negative continuous potential on $M_n$ with maximum norm \[
\limsup_{n\to\infty} \|V_n\|_{\infty} < \lambda. 
\]
such that the pair $(M_n,V_n)$ Benjamini-Schramm converges to $(\Hd,V)$ where $V$ is a stationary random function on $\Hd$. Then the spectral gap of the foliated space $(\Hd,V)$ is at least $\lambda$ and the $Q$-processes of $(M_n,V_n)$ converge to the $Q$-process of $(M,V)$.
\end{theorem}

We need to explain Benjamini-Schramm convergence of $(M_n,V_n)$ to $(M,V)$ which in the above case is quite simple (for a general definition see \cite{abert2018eigenfunctions}). Here, we lift the function $V_n$ to $\Hd$ using a random root and frame, which gives us a periodic, stationary random function $\widetilde V_n$ on $\Hd$. Benjamini-Schramm convergence means weak convergence of the law of the function $V_n'$ to the law of $V$. We will clarify the exact space and setup in Section \ref{sec:section5}. 

It is not clear, however, which stationary random potentials can be approximated by periodic ones with a good spectral gap. This is the point where factor of Poisson Point Process potentials shine against an arbitrary stationary random potential. Indeed, any factor of PPP potential can be approximated by periodic potentials which allows us to prove a general stability result there. 

\begin{theorem}\label{thm:thm3}
There exists $\lambda_d > 0 \; (d \geq 2)$ such that the following hold. Let $\omega$ be a homogeneous Poisson Point Process (PPP) on $\Hd$ and let
$V_\omega:\Hd\to[0,\infty)$ be a continuous factor of $\omega$ potential with maximum norm 
\begin{equation}
  \|V_\omega\|_\infty < \lambda_d. 
\end{equation}
Then the Q-process of $(\Hd,V)$ can be approximated by periodic Q-processes, that is, Q-processes defined by periodic potentials. 
\end{theorem}

The above theorem has two parts. First, we use that for every $d \geq 2$ there exists an expander family of compact hyperbolic $d$-manifolds that Benjamini-Schramm converges to $\Hd$. This follows from two results: 1) By a well-known result of Clozel \cite{clozel}, for any symmetric space $X$ for a semisimple Lie group (including $\Hd$), there exists a constant $\epsilon_d > 0$, depending only on $X$, such that for every congruence lattice $\Gamma$ in $G$, the spectral gap of the orbifold $\Gamma \backslash X$ is at least $\epsilon_d$. 2) Fixing any co-compact arithmetic lattice $\Gamma$ on $\Hd$, by \cite{7sam}, we have that \emph{any} sequence of congruence subgroups of $\Gamma$, $\Gamma \backslash \Hd$ Benjamini-Schramm converges to $\Hd$. We set $\lambda_d$ to be the best asymptotic spectral gap for which this can be achieved. With the exception of $d=2$, the value of $\lambda_d$ is not known: for $d=2$, by the recent work of Hide and Magee \cite{hide2023near}, we have $\lambda_2 = 1/8$ which is the optimum. 

Second, we use that for \emph{any} such sequence of manifolds, the PPP of the compact manifolds will converge to the PPP of $\Hd$, and this extends to any factor potential. 
Together these yield a sequence of $(M_n,V_n)$ that satisfies the assumptions of Theorem \ref{thm:thm2}. 

A limiting case of the above model is the so-called \emph{hard obstacle model},
where the particle is conditioned to avoid an $r$-neighborhood of the PPP
$$\{ x \colon d(x,\omega) < r \} = \bigcup_{y\in \omega} B(y,r).$$
If the radius is sufficiently small, the complement of the $r$-neighborhood contains an infinite cluster with positive probability \cite{benjamini2001percolation,tykesson2007number}, raising the following question.

\begin{question}
On an infinite connected component of the supercritical hard Poisson obstacle model,
does the limiting $Q$-process exist?
\end{question}

Formally, the hard obstacle corresponds to the limit $V_{\max}\to\infty$ in the soft potential $ V_\omega(x) = V_{\max}\,\mathbf{1}_{\{ d(x,\omega) < r \}}$. At the moment, the method we developed here heavily use the bounded assumption of the potential and we are not able to answer this question.

\bigskip  

\noindent \textbf{Acknowledgements.} The authors thank Francois Ledrappier and Jean Raimbault for helpful discussions. 

\section{Preliminaries}
\paragraph{Hyperbolic space.}
We consider the hyperboloid model:
\[
\Hd
=
\left\{
z=(z_0,z_1,\ldots,z_d)\in \dR^{1+d}
\;\middle|\; z_0>0, \;
\langle z,z\rangle_J=-1
\right\},
\]
where $J=\mathrm{diag}(-1,1,\ldots,1)$ and $\langle\cdot,\cdot\rangle_J$ denotes the associated Minkowski bilinear form. Additional structure will be defined relative to a positively oriented orthonormal frame $(o,u)\in\FHd$. We fix the canonical choice $o=e_0\in\Hd$ and $u=(e_1,\ldots,e_d)\in (T_o \Hd )^d$, where $(e_0,\ldots,e_d)$ denotes the standard basis of
$\dR^{1+d}$. In this model, the group of orientation-preserving isometries of $\Hd$
coincides with the identity component of the group of real linear
transformations preserving the form, that is, $A^TJA = J$. We denote it by
$$
G=\mathrm{Isom}^{+}(\Hd)=\mathrm{PSO}(1,d)\subset GL_{d+1}(\dR). 
$$
 Any isometry $g\in G$ acts naturaly on the bundle of (positively-oriented) orthonormal frames by
\[
g\cdot(o,u) := (g \cdot o,\, \dd g_o(u)).
\]
Since this action is simply transitive, we obtain the identification $G \simeq \FHd$. We denote by $K$ the stabilizer of the origin $o$, so that
\[
\Hd \simeq G/K.
\]
We refers to \cite{franchi2012hyperbolic} for more on the hyperboloid model.

\paragraph{Brownian motion among soft traps.}
Let $\bP^x$ denote the Wiener measure on $C(\dR_{\geq 0},\Hd)$ of the Brownian motion
$(X_t)_{t\ge0}$ with $X_0=x$ and generator $\tfrac12\Delta_\Hd$, and let $\bE^x$ be the associated expectation. For $f\in C_c^\infty(\Hd)$,
\[
\bE^x[f(X_t)]
=
\bigl(e^{\frac t2\Delta_\Hd}f\bigr)(x).
\]
Given a bounded continuous potential $V:\Hd\to\dR$, we define the tilted
path measure $\mathbf Q_T^{x}$ by
\[ \frac{\dd \mathbf Q_T^{x}}{\dd \bP^{x}}
=
\frac{1}{Z_T^{x}}
\exp\!\left(- \int_0^T V(X_s)\,\dd s \right),
\]
\begin{equation}\label{eq:Ztx}
Z_T^{x}
= \bE^x\!\left[
\exp\!\left(- \int_0^T V(X_s)\,\dd s \right)
\right].
\end{equation}

Assume that that $(\mathbf Q_T^{x})_{T>0}$ converges weakly on $C(\dR_{\geq 0},\Hd)$ endowed with the topology of uniform convergence on compact time intervals. More precisely, we assume that for every $t\ge0$ and every Borel set $B \in \mathcal{F}_t$ (the canonical filtration),
\[
\mathbf Q^{x}(B)
:=
\lim_{T\to\infty}\mathbf Q_T^{x}(B)
\]
exists. These finite-dimensional marginals define a unique Markov probability
measure $\mathbf Q^{x}$, called the \emph{$Q$-process} associated with the killing potential $V$.

\paragraph{The periodic case.}
We say that $V$ is \emph{periodic} if it is invariant under a torsion-free
cocompact lattice $\Gamma\le G$, so that
$M=\Gamma \backslash \Hd$ is a closed hyperbolic manifold. The Schr\"odinger operator
\[
H=-\tfrac12\Delta_M+V
\]
defines a self-adjoint operator on $L^2(M,\operatorname{vol}_M)$ with
discrete spectrum. Its lowest eigenvalue $\rho$ is simple
\cite{Shigekawa1987}, with a positive $\varphi$ (normalized) eigenvector. Let $T_\varphi$ be the diagonal operator acting by multiplication by $\varphi$. The conjugated operator (see \cite[Section 1.15.8]{bakry2014analysis})
\begin{equation}\label{eq:doobtransform}
L = T_\varphi^*(\rho-H)T_\varphi
=
\tfrac12\Delta_M + \nabla\log\varphi\cdot\nabla    
\end{equation}
generates a diffusion process whose transition density $p_L(t,x,y) = e^{tL}_{xy}$ with respect to
$\operatorname{vol}_M$ is
\[
p_L(t,x,y)
=p_{\tfrac{1}{2}\Delta_M}(t,x,y)
e^{t\rho}
\bE^x\!\left[
\left.\exp\!\left(-\int_0^t V(X_s)\,\dd s\right)\right| X_t=y
\right] \frac{\varphi(y)}{\varphi(x)},
\]
where $\bE^x$ denotes expectation with respect to Brownian motion started from $X_0 = x\in M$. We refers to \cite{Shigekawa1987,Polymerakis2019,BallamannPolymerakis}; and \cite{Sullivan1987,Brooks1985} for renormalization by positive harmonic functions. Lifting this diffusion to $\Hd$ yields the existence of the $Q$-process in the periodic setting.

\paragraph{Casimir operator.}
In the hyperboloid model, the Lie algebra $so(1,d)\simeq T_{(o,u)}\FHd$ is generated by the orthonormal basis
$$
E_k := e_0 e_k^{*}-e_k e_0^{*}, \qquad
E_{i,j} := e_i e_j^{*}-e_j e_i^{*},
\quad k,i,j\in\{1,\ldots,d\},\ i<j. 
$$
For $A\in so(1,d)$ and $F\in C^\infty(G)$, we recall the (right) Lie derivative
\[
\mathcal{L}_A F(g)
:=
\left.\frac{\dd}{\dd t}\right|_{t=0} F\bigl(g\exp(tA)\bigr).
\]
The \emph{Casimir operator} $\Delta_G$ on $G$ is the bi-invariant second-order
differential operator
\[
\Delta_G
=
\sum_{k=1}^{d} \mathcal{L}_{E_k}^2
-
\sum_{1\le i<j\le d} \mathcal{L}_{E_{i,j}}^2 .
\]
It is essentially self-adjoint on $C_c^\infty(G)$ with respect to the Haar measure on $G$, has purely continuous spectrum $\spec(-\Delta_G) = [\tfrac{(d-1)^2}{4} ,\infty)$, see \cite[ch. VIII]{knapp2001representation}. The Casimir operator restricts on right $K$-invariant functions to the usual Laplace--Beltrami operator on the hyperbolic space. For $F \in C^\infty(G)$ of the form $F=f\circ\pi$, with $f\in C^\infty(\Hd)$ and $\pi:G\to G/K$ is the quotient map, one has
\begin{equation}\label{eq:CasimirToLB}
(\Delta_G F)(g)= (\Delta_\Hd f)(gK).
\end{equation}
Let $(g_t)_{t\ge0}$ denote the Brownian motion started at $g_0 =e_G$ the identity, generated by $\tfrac12\Delta_G$. Then $(g_t\cdot o)_{t\ge0}$ has the same law as $(X_t)_{t\ge0}$ with $X_0 = o$.

\subsection{The space of configurations}

We denote by
\begin{equation}\label{eq:conf}
\Omega
=
\left\{
\omega\subset\Hd \;\middle|\; \omega \text{ is locally finite}
\right\}    
\end{equation}
the configuration space, endowed with the topology generated by the counting
maps
\[
\omega \longmapsto |\omega\cap C|,
\qquad C\subset\Hd \text{ compact}.
\]
The law of the \emph{Poisson point process} on $\Hd$ with intensity
$\vol_\Hd$ is denoted by $\mu$. It is the unique Borel probability measure on $\Omega$ such that:
\begin{itemize}
    \item for every compact set $C\subset\Hd$,
    the random variable $|\omega\cap C|$ has distribution
    $\operatorname{Poiss}(\vol_\Hd(C))$;
    \item for any pair of disjoint compact sets $C_1,C_2\subset\Hd$,
    the random variables $|\omega\cap C_1|$ and $|\omega\cap C_2|$
    are independent.
\end{itemize}
Since $G$ preserves the volume, its action on $\Omega$ given by
\[
g\cdot \omega := \{g x \mid x \in \omega\}
\]
preserves the measure $\mu$.

\begin{definition}\label{def:FactorOfPPP}
A \emph{factor of the Poisson point process} is a random potential
$V_\omega:\Hd\to\mathbb R$ of the form
\[
V_\omega(x) = V_\Omega(g^{-1}\cdot \omega),
\qquad g \in G \text{ such that } g\cdot o = x,
\]
where $V_\Omega:\Omega\to\mathbb R$ is a bounded Borel function which is
$K$-invariant, i.e.\ $V_\Omega(k \cdot \omega)=V_\Omega(\omega)$ for all $k\in K$.
\end{definition}

\begin{example} The distance from $x$ to the PPP
\[
d(x,\omega) := \min_{y\in\omega} d(x,y).
\]
is clearly a $K$-invariant function on $\Omega$. Given any smooth decreasing function $\eta:\dR_{\ge0}\to\dR$ with, the potential
\[
V_\omega(x) = \eta(d(x,\omega))
\]
is a factor of the Poisson point process satisfying the assumptions of Theorem~\ref{thm:thm1} if $\eta(0)<(d-1)^2/8$.
\end{example}

Our main approach is to study square-integrable factors of the Poisson point
process $L^2(\Omega,\mu)$, endowed with the foliated structure described below.

The configuration space $\Omega$ carries a natural foliated structure induced by the action of the Lie group. For $\omega\in\Omega$, the $G$-orbit
\[
\lL_\omega := \{ g\cdot \omega \mid g\in G \} \subset \Omega
\] 
defines the leaf through $\omega$.
\begin{remark}\label{rmk:GActfreelyOnPPP} Consider the Borel set of configuration with trivial stabilizer
$$\Omega_0 =  \{ \omega \mid  \operatorname{Stab}_G(\omega)=\{e_G\}\}.$$
Since $\Hd$ has infinite volume, one has $\mu(\Omega_0)=1$; see Section~4.2 of \cite{AbertBiringer2022}. When $\omega \in \Omega_0$, the leaf $\lL_\omega$ is canonically diffeomorphic to $G$. In the sequel, when referring to a leaf, we implicitly assume $\omega \in \Omega_0$.
\end{remark}

We refer to \cite{moore2006global,candel2000foliations} for general background on foliated spaces. Considering the base frame, the triple $(\Hd, (o,u), \omega)$ corresponds to the setting of random manifold decorated by a PPP as considered in the \textit{desingularization Theorems} in \cite{abert2025uniform}. 

The foliation can be interpreted as observing the same configuration from
different reference frames. This leads to a natural correspondence between functions on the foliated space $\Omega$ and random functions on $G$.

Given a measurable function $F\colon\Omega\to\mathbb R$, we associate a random function on $G$ with respect to the probability space $(\Omega,\mu)$:
\begin{equation}\label{eq:FoliationToRandomWave}
F_\omega(g) := F(g^{-1}\cdot\omega),
\qquad g\in G,
\end{equation}
Thus, $F_\omega(g)$ represents the value of $F$ evaluated on the configuration
$\omega$ as seen from the frame $g\cdot(o,u)$. If $F$ is $K$-invariant, since $\Hd \simeq G/K$ we obtain a random function $F_\omega \colon 
\Hd \to \dR$.

\section{Foliated Brownian motion.}

The \emph{foliated Brownian motion},  as introduced by Garnett \cite{Garnett1983}, is the Markov process whose trajectories evolve along the leaves of the foliation according to Brownian
motion. This define a contraction semi-group on $L^2(\Omega,\mu)$
\[
\bigl(e^{\frac{t}{2}\Delta_\Omega}F\bigr)(\omega)
=
\bE^{e_G}\!\bigl[F(g_t^{-1}\cdot\omega)\bigr],
\]
where $(g_t)_{t\ge0}$ denotes Brownian motion on $G$ started from the identity
$e_G$. The generator $\Delta_\Omega$ is called the \emph{foliated Laplacian} and verifies on function that are leaf-wise $C^2$ the identity 
\[
(\Delta_\Omega F)(\omega)
:=
(\Delta_G F_\omega)(e_G),
\]
where $F_\omega(g)=F(g^{-1}\cdot\omega)$. The fact that the operator $\Delta_\Omega$ is symmetric  when $\Omega$ is endowed with the Poisson measure $\mu$  follows from the general theory of measured foliation, since this is a completely invariant transverse measure for the foliation. Equivalently, the random decorated manifold with its canonical frame is unimodular in the sense of \cite[Theorem~3.1]{abert2025uniform}. As the generator of a symmetric contraction semi-group, the operator $\Delta_\Omega$ is self-adjoint. In fact, much more can be said using representation theory.

\begin{theorem}[The PPP is Ramanujan]\label{thm:PPPisRamanujan} 
The $G$-action on the $L^2(\Omega,\mu)$, the foliation of the PPP, decomposes as a direct sum of the trivial representation and a countable sum of copies of the regular representation. As a consequence, the foliated Laplacian has spectrum 
\begin{equation}\label{eq:spectralgap}
  \spec(-\Delta_\Omega)= \{0\} \cup \spec(-\Delta_G) = \{0\} \cup \left[\frac{(d-1)^2}{4}, \infty\right).  
\end{equation}
The bottom eigenvalue is simple and associated with the constant function $\mathbf 1_{\Omega}$.
\end{theorem}

\begin{remark}
This is the continuous analogue of
\cite[Theorem~2.1 and Corollary~2.2]{lyons2011perfect} for the Bernoulli graphing of a discrete group.
\end{remark}

\begin{proof} 
We use the Fock representation of the Poisson process as constructed by G.Last and M.Penrose in \cite{LastPenrose2011}, see also \cite{Albeverio1998} for more on this approach. Consider the symmetric Fock space
\[
\Gamma_s(\Hd)
:=
\mathbb C \oplus \bigoplus_{n=1}^\infty \mathcal H^{\odot n},
\]
where $\mathcal{H} := L^2(\mathbb H^d,\mathrm{vol}_{\mathbb H^d})$ and $\mathcal{H}^{\odot n}$ denote the symmetric function in $n$ variables in $\Hd$, equipped with the norm
\[
\|f\|_{\Gamma_s(\Hd)}^2
=
|f_0|^2
+
\sum_{n=1}^\infty
n!\,\|f_n\|_{L^2((\mathbb H^d)^n)}^2.
\]
Given $F\in L^2(\Omega,\mu)$, we define the map $U : F \mapsto f=(f_n)_{n\ge0}$, given by \cite{LastPenrose2011}
\[
f_0 := \int_\Omega F(\omega) \mu(\dd \omega),
\qquad
f_n(x_1,\dots,x_n)
:=
\sum_{J\subset \{1, \ldots n\}} (-1)^{n-|J|}\int_\Omega F(\omega \, \cup \, \{x_j \colon j \in J\} ) \mu(\dd \omega),
\]
The main result in \cite{LastPenrose2011} is that $(f_n)_{n\ge0}$ belongs to $\Gamma_s(\Hd)$ and $U$ is in fact an isometry
$$
\| F \|_{L^2(\Omega,\mu)} 
=
\|f\|_{\Gamma_s(\Hd)}= \sqrt{
\sum_{n=0}^\infty
n!\,\|f_n\|_{L^2((\mathbb H^d)^n)}^2}.
$$
So far, we have only used the measure space structure of $\Hd$. We now observe that the $G$-action acts diagonally on each summand of $\Gamma_s(\Hd)$ via the unitary transform $U$. The constant function $1_\Omega$ is sent to the first component $f_0=1$ and corresponds to the trivial representation. For $f_n, n\geq 1$, we see that $F(g\cdot \omega)$ translate to $f_n(g\cdot x_1, \ldots, g \cdot x_n)$, so that each factor $n\geq1$ corresponds to the regular representation. Since the foliated Laplacian is defined by the action through the Brownian motion on $G$, 
$$e^{\frac{t}{2}\Delta_\Omega}F(\omega) = \bE^{e_G} [F(g_t\cdot \omega)],$$
we obtain that $\spec(-\Delta_\Omega)= \{0\} \cup \bigcup_{n \geq 1} \spec(-\Delta_G)$.
\end{proof}

\section{Proof of Theorem~\ref{thm:thm1}}

In this section, we use the membrane method introduced in
\cite{abert2025uniform} to prove the existence of the limiting $Q$-process
claimed in Theorem~\ref{thm:thm1}.
Since the potential is a factor of a Poisson point process,
its law is invariant under the action of $G$.
As a consequence, the random Schr\"odinger operator
$H_\omega = -\tfrac{1}{2}\Delta_\Hd + V_\omega$ does not admit a ground state
$\varphi_\omega\in L^2(\Hd)$.
This phenomenon already occurs for the free Laplacian, whose constant
eigenfunction $\mathbf 1_\Hd$ belongs to $L^\infty(\Hd)$ but not to
$L^2(\Hd)$.

By contrast, on the foliated space associated with the Poisson point process, the constant function $\mathbf 1_{\Omega}$ is a genuine $L^2(\Omega,\mu)$-eigenvector of the foliated Laplacian $\Delta_\Omega$. Restricted to a leaf $\lL_\omega$, it reduces to the constant function $\mathbf 1_\Hd$ on $\Hd$ (seen as a $K$-invariant constant function on $G$).

Spectral perturbation theory implies that, under a sufficiently small bounded potential, the corresponding foliated Schr\"odinger operator still admits a
positive ground state $\varphi\in L^2(\Omega,\mu)$.
Via the correspondence of equation \eqref{eq:FoliationToRandomWave}, this yields a
random positive generalized eigenfunction $\varphi_\omega$ on $\Hd$ for, which can be used as the Doob $h$-transform to construct the quenched $Q$-process.

\begin{proof}[Proof of Theorem \ref{thm:thm1}]
Consider a random potential $V_\omega(\cdot)$, as factor of Poisson. Following definition \ref{def:FactorOfPPP}, we obtain a $K$-invariant potential function $V_\Omega \in L^{\infty}(\Omega)$, and denote $V$ associated multiplication operator $V \colon F\mapsto V_\Omega\cdot F$. As bounded perturbation of self-adjoint operator, the foliated Schr\"odinger operator 
\begin{equation}\label{ref:FoliatedSCH}
    H = -\tfrac{1}{2}\Delta_\Omega + V  
\end{equation}
is self-adjoint. It acts on a leaf $L^2(\mathcal{L}_\omega) \simeq L^2(G)$, by  $H_G = -\tfrac{1}{2}\Delta_G + V_\omega$. We use classic perturbation theory \cite[chapter VII]{reed1978iv} to show that $H$ admits a lower eigenvector. 

We assume first that $\| V \| < \frac{1}{2}\frac{(d-1)^2}{8}$, allowing the potential to take negative values. From the spectral estimate of equation \ref{eq:spectralgap} and for $z$ in the contour $\gamma=\{z \colon |z|=\frac{1}{2}\frac{(d-1)^2}{8}\}$, the Born series 
\begin{equation}\label{eq:bornserie}
(H-z)^{-1} = (\tfrac{-1}{2}\Delta_\Omega-z)^{-1}\sum_k (-V(\tfrac{-1}{2}\Delta_\Omega-z)^{-1})^k  
\end{equation}
is absolutely convergent, showing that the spectrum is separated away from the range of $\gamma$. The series is analytic in the perturbation $V$, and so does the orthogonal projection
$$P_V = \frac{1}{2\pi i} \int_{\gamma} (H-z)^{-1} \dd z.$$
For $V=0$, $P_0$ is the projection onto the subspace $1_\Omega$, by Theorem \ref{thm:PPPisRamanujan}. Since the range $\dim \operatorname{Range} P_V$ is continuous (as a trace), it is constant and $P_V$ project to the dimension one eigen-space associated with the ground state.
We set 
$$\psi = \frac{P_V1_\Omega}{\| P_V1_\Omega \|}, \qquad \rho =\frac{\| H P_V1_\Omega \|}{\| P_V1_\Omega\|}.$$
Note that the semi-group $e^{tH}$ is positivity preserving, so that the ground state has constant sign and therefore  $P_V1_\Omega \not=0$. In fact, it is not hard to see that $\psi$ must is positive almost everywhere. Both the eigenvalue and the eigenvector varies analytically in $V$ and are well defined up to $\| V\| < \tfrac{1}{2} \frac{(d-1)^2}{8}$. The equation $H \varphi = \rho \varphi$ translate on the leaves $\mathcal{L}_\omega$ to
$$ -\frac{1}{2}\Delta_G\varphi_\omega + V_\omega\varphi_\omega = \rho \varphi_\omega.$$
Let $Z_t = e^{tH} 1_\Omega \in L^2(\Omega)$. Note that $1_\Omega, V_\Omega$ and therefore $Z_t$ are all $K$-invariant function, and one has
\begin{align*}
   Z_T(\omega) &= \bE^{e_G}\left[\exp\left(-\int_0^T V(g_t \cdot \omega)\dd t \right)\right] \\
   &= \bE^{o}\left[\exp\left(-\int_0^T V_\omega(X_t) \dd t \right) \right] = Z_{\omega,T}^o,\qquad \omega\, \mu-a.e.
\end{align*}
where the first expectation is over the Brownian motion on G while the second is on $\Hd$ starting at the base point of the canonical frame. More generally, one has $Z_T(g^{-1}\cdot \omega) = Z_{\omega,T}^x$ for any $x\in \Hd, g\in G$ with $g\cdot o = x$. The quantity $\epsilon := \frac{(d-1)^2 }{8}- 2\| V\|$ lower bound the spectral gap, $t_n = 2\ln(n)/\epsilon$, one has 
$$\| e^{-t_n\rho }Z_t - \langle \psi, 1_\Omega \rangle \psi \| \leq e^{-t_n\epsilon} = O\left(\frac{1}{n^2}\right),$$
$$\sum_n \| e^{-t_n \rho }Z_{t_n} - \langle \psi, 1_\Omega \rangle \psi \| < \infty.$$
So that $(e^{-t_n \rho }Z_{t_n})_n$ converges $\mu$-almost everywhere. For $t_n \leq t < t_{n+1}$, we get 
$$ e^{-t\rho}Z_t(\omega) = e^{-t\rho}\bE^{e_G}[\exp(\int_o^t-V(g_s^{-1}\cdot \omega) \dd s)] \leq e^{-t_n\rho } Z_{t_n}(\omega)e^{-(t-t_n)\rho} \exp(\|V \|_{\infty}(t-t_n)).$$ 
This implies the existence of the limit as $t \to \infty$ on a set of full $\mu$-measure. We obtain by continuity the convergence of the normalizing constant $Z_{\omega,T}^x$, simultaneously of all $x$,. Once the convergence of the normalizing constants (as defined by equation \eqref{eq:Ztx})  is proven, the existence of the limiting $Q$-process is classical. For a Borel set of the form $B = \{X_t \in A\} \in \mathcal{F}_t$, and $t<T$, we obtain 

\begin{align*}
Q_{\omega,T}^o(B)&=\frac{1}{Z_{\omega,T}^o}\bE^o\!\left[\exp\!\left(-\int_0^T V_\omega(X_s)\,ds\right)\mathbf 1_{\{X_t\in A\}}\right]
\\
&=\bE^o\!\left[\exp\!\left(-\int_0^t V_\omega(X_s)\,ds\right)\mathbf 1_{\{X_t\in A\}}
\,\frac{Z_{\omega,T-t}^{X_t
}}{Z_{\omega,T}^o}\right] \\
&=\int_A p_t(o,y)\,\bE^{o}\!\left[\exp\!\left(-\int_0^t V_\omega(X_s)\,ds\right)\,\Big|\,X_t=y\right]\,
\frac{Z_{\omega,T-t}^y}{Z_{\omega,T}^o}\,\mathrm{vol}_{\mathbb H^d}(\dd y)\\ 
&\xrightarrow[T \to \infty]{} \int_A p_t(o,y)\, \bE^{o}\!\left[\exp\!\left(-\int_0^t V_\omega(X_s)\,ds\right)\,\Big|\,X_t=y\right]\,
\frac{e^{t\rho}\varphi_\omega(y)}{\varphi_\omega(o)}\,\mathrm{vol}_{\mathbb H^d}(\dd y) 
\\
&=Q_\omega^o(B) 
\end{align*}

This shows the existence of the limiting $Q$-process. The limiting generator has the same form as in \eqref{eq:doobtransform}, 
$$L_\omega := T_{\varphi_\omega}^* (\rho-H_\omega )T_{\varphi_\omega} =\tfrac{1}{2} \Delta_\Hd +  \nabla \log\varphi_\omega\cdot\nabla.$$ 
So far, we have proved the existence of the $Q$-process under the condition $\|V_\omega\|_\infty < \tfrac{1}{2}\tfrac{(d-1)^2}{8}$, without assuming that $V$ is non-negative. More generally, let
\[
V_- := \inf_\Hd V_\omega,
\qquad
V_+ := \sup_\Hd V_\omega,
\qquad
\overline V := \frac{V_- + V_+}{2}
\]
denote the extrema and the midrange of the potential. Assume that
\[
\operatorname{osc}_\Hd V
:= V_+ - V_-
< \frac{(d-1)^2}{8}.
\]
The above argument applies to the shifted generator $-\tfrac12\Delta_\Omega + \overline V$ with a centered perturbation $ \|V_\omega - \overline V\|_\infty < \tfrac{1}{2}\tfrac{(d-1)^2}{8}$, this concludes the proof of Theorem \ref{thm:thm1}.
\end{proof}

\section{Benjamini--Schramm limits and periodization}\label{sec:section5} 

We now turn to the stability of the $Q$-process. The appropriate notion of
closeness is given by the local convergence introduced by
Benjamini--Schramm in \cite{BenjaminiSchramm2001}. We recall this notion in the
case of decorated manifolds; see \cite{AbertBiringer2022}.

\begin{definition}\label{def:BS}
Let $(M_n,V_n)$ be a sequence of finite-volume Riemannian $d$-manifolds
equipped with continuous potentials $V_n\colon M_n\to\mathbb R$.
We say that $(M_n,V_n)$ converges in the \emph{Benjamini--Schramm sense}
(BS) to a random pointed hyperbolic manifold with potential $(M,V,o)$ if
\[
(M_n,V_n,x_n)
\;\xrightarrow[n\to\infty]{\mathrm{law}}\;
(M,V,o),
\]
where $x_n$ is chosen according to the Riemannian volume measure.
The convergence in law is understood with respect to the pointed $C^\infty$-topology on
manifolds together with local uniform convergence of the potentials.
In this case, we write
\[
(M_n,V_n) \xrightarrow{BS} (M,V,o).
\]
\end{definition}

In this paper, we only consider hyperbolic manifolds whose limit is the
hyperbolic space. In this case,
$M_n \xrightarrow{BS} (\Hd,o)$ means that the injectivity radius at a
random point of $M_n$ tends to infinity in probability.
Before proving Theorem~\ref{thm:thm2}, we give an explicit construction of the
foliated space that appears in the BS limit. This is achieved by lifting the
sequence to random copies of $\Hd$ and taking weak limits, after breaking
possible symmetries by adding an independent Poisson point process, in the
spirit of \cite{AbertBiringer2022}.

We define the space of \emph{augmented configurations} by
\begin{equation}\label{eq:augmentedconf}
\hOmega
=
\lV \times \Omega
=
\left\{
\hat{\omega}=(V,\omega)
\;\middle|\;
V\colon \Hd\to\dR \text{ continuous, and }
\omega\subset\Hd \text{ locally finite}
\right\}.
\end{equation}
Here $\lV=C(\Hd,\dR)$ is endowed with the topology of uniform convergence on
compact sets. The space $\hOmega$ is a Polish space.
The group $G$ acts diagonally on $\hOmega$ by
\[
g\cdot \hat{\omega}
:=
(g\cdot V,\, g\cdot \omega),
\]
where $(g\cdot V)(x):=V(g^{-1}\cdot x)$.
The induced foliation is defined by the $G$-orbits as before. Given a sequence $(M_n,V_n)$ of hyperbolic $d$-manifolds with continuous
potentials, we choose a random orthonormal frame $(o_n,u_n)$ and consider the
covering map $\pi_n\colon \Hd \to M_n$ sending the canonical frame $(o,u)$ to
$(o_n,u_n)$. This yields a random lifted potential on $\Hd$
\begin{equation}\label{eq:fromMntoHd}
  \widetilde V_n := V_n \circ \pi_n.  
\end{equation}
Let $\widehat{\mu}_n$ be the probability measure on
$\hOmega$ given by the joint law of
$(\widetilde V_n,\omega)$, where $\omega$ is an independent Poisson point
process on $\Hd$.
If $(M_n,V_n) \xrightarrow{BS} (\Hd,V,o)$, then the measures weakly converge, the limit
$$\hat{\mu} := \lim_n \hat{\mu}_n$$
defines the limiting measured foliation. As in \ref{rmk:GActfreelyOnPPP}, for
$\hat{\mu}$-almost every $\hat{\omega}$, the leaf
$\mathcal{L}_{\hat{\omega}}$ is diffeomorphic to $G$.
The fact that the foliated Laplacian $\Delta_{\hOmega}$ is symmetric, when $\hOmega$ is endowed with a measure $\hat{\mu}$ follows from unimodularity, see \cite{AbertBiringer2022}. With this construction at hand, we are ready to prove Theorem~\ref{thm:thm2}.

\begin{proof}[Proof of Theorem \ref{thm:thm2}]
Consider the sequence of Schr\"odinger operators on $L^2(M_n,\vol_{M_n})$,
\[
H_n = -\tfrac{1}{2}\Delta_{M_n} + V_n,
\]
and let $\varphi_n$ denote the normalized positive eigenfunction associated with the lowest eigenvalue $\rho_n$. On each $M_n$, the $Q$-process is generated by the Doob transform $L_n$ associated with $\varphi_n$, see \eqref{eq:doobtransform}. Since $(M_n)$ is an expander sequence, the spectral gap of $-\tfrac12\Delta_{M_n}$ is uniformly bounded from below.
Together with the assumption $\limsup_n\|V_n\|_\infty<\lambda$, this implies that
the limiting random potential $V_\omega$ satisfies the assumptions of
Theorem~\ref{thm:thm1}. As in \eqref{ref:FoliatedSCH} we denote by
\[
H = -\tfrac{1}{2}\Delta_{\hOmega} + V
\]
the limiting foliated Schr\"odinger operator on $L^2(\hOmega,\hat{\mu})$. We denote the spectral gap of the foliated Laplacian
$$\lambda_1(\tfrac{-1}{2}\Delta_{\hOmega}):=\inf (\spec \left(\tfrac{-1}{2}\Delta_{\hOmega} \right) \setminus \{0\}).$$
Note that $\lambda_1(\tfrac{-1}{2} \Delta_{\hOmega})$ depends implicitly on the measure $\widehat{\mu}$, and that it need not correspond to an eigenvalue in the presence of continuous spectrum. We claim that this quantity is at least the spectral gap of the sequence, namely
\[
\lambda := \liminf_{n\to\infty} \lambda_1(M_n).
\]
This lower semicontinuity result is classical in the spectral theory of graphings \cite{abert2025uniform}, and we adapt it here to the language of foliations. By Weyl's criterion, one can find a sequence $\psi_k$ of functions orthogonal to $\mathbf{1}_{\hOmega}$ such that
\[
-\tfrac{1}{2}\langle \psi_k , \Delta_{\hOmega} \psi_k \rangle
=
\lambda_1(\tfrac{-1}{2}\Delta_{\hOmega}) + o(1).
\]
By density, we may assume that the $\psi_k$ are smooth in the leaf direction and is a factor of the a finite neighborhood of the scenery $\hat{\omega}$ as seen from the base point $o$. That is, for some sequence slowly increasing $R_k \to \infty$, the value $\psi_k(\widehat{\omega})$ depends only on the restriction
\[
\hat{\omega}\big|_{B(o,R_k)}
=
\bigl(V|_{B(o,R_k)},\, \omega \cap B(o,R_k)\bigr).
\]
But such function $\psi_k$ can be copied on the the foliation $(\hOmega,\hat{\mu_n})$, and on most point (where the injectivity radius is larger than $R_k$) to the compact surfaces $M_n$ following the reverse construction given by equation \eqref{eq:fromMntoHd}. This implies that $ \lambda_1(\tfrac{-1}{2}\Delta_{\hOmega}) \geq \lambda := \liminf_{n\to\infty} \lambda_1(\tfrac{-1}{2}M_n).$.

Now, since $\|V\|_\infty \le \limsup\|V_n\|_\infty < \lambda$, we can adapt the proof of Theorem~\ref{thm:thm1} to the foliation $\widehat{\Omega}$ and deduce that the limiting $Q$-process exists. Let $\varphi\in L^2(\widehat{\Omega},\widehat{\mu})$ denote the positive ground state of the limiting foliated Schr\"odinger operator $H$. The corresponding eigenvalue $\rho$ is separated away from the rest of the spectrum by at least $\epsilon = \tfrac{1}{2}(\lambda - \limsup\|V_n\|_\infty)$. For $\widehat{\mu}$-almost every $\widehat{\omega}$, the limiting $Q$-process on the leaf $\mathcal{L}_{\widehat{\omega}}$ is given by the Doob transform associated with the eigenwave $\varphi_{\widehat{\omega}}(\cdot)$, viewed as a function on $\Hd$. 

It remains to prove that the $Q$-processes on $M_n$ converge to this limiting process.  
Let $\nu_{H_n}^{\mathbf{1}_{M_n}}$ be the spectral measure of $H_n$ associated with the normalized constant function, and let $\nu_{H}^{\mathbf{1}_{\widehat{\Omega}}}$ denote the analogous quantity for the limiting foliation. One has
\begin{align*}
\frac{1}{\operatorname{Vol} M_n }
\langle \mathbf{1}_{M_n}, e^{-t H_n} \mathbf{1}_{M_n} \rangle
&=
\int e^{-t\lambda}\,\nu_{H_n}^{\mathbf{1}_{M_n}}(\dd\lambda) =
\int_{M_n}
\bE^{x_n}\!\left[\exp\!\left(-\int_0^t V_n(X_s)\,\dd s\right)\right]
\frac{\vol_{M_n}(\dd x_n)}{\operatorname{Vol} M_n},
\\[1ex]
\langle \mathbf{1}_{\widehat{\Omega}}, e^{-t H} \mathbf{1}_{\widehat{\Omega}} \rangle
&=
\int e^{-t\lambda}\,\nu_{H}^{\mathbf{1}_{\widehat{\Omega}}}(\dd\lambda)=
\int_{\widehat{\Omega}}
\bE^{o}\!\left[\exp\!\left(-\int_0^t V(X_s)\,\dd s\right)\right]
\widehat{\mu}(\dd \widehat{\omega}),
\end{align*}
where in the last integral $\widehat{\omega}=(V,\omega)$.

By Benjamini--Schramm convergence, the pointed manifolds $(M_n,x_n)$ have injectivity radius tending to infinity in probability, and the local laws of the potentials converge to that of the limiting stationary potential $V$. Consequently, for each fixed $t>0$, the Laplace transforms converge pointwise. By the L\'evy continuity theorem for Laplace transforms \cite[Ch.~XIII]{feller1991introduction}, we obtain weak convergence of the spectral measures.

By the spectral gap assumption, the ground eigenvalue $\rho_n \in [0,\| V_n \|_\infty]$ is isolated from the remainder of spectrum of $H_n$ contained in $[\lambda,\infty)$. It follows that
\[
\rho_n \longrightarrow \rho,
\qquad
\nu_{H_n}^{\mathbf{1}_{M_n}}(\{\rho_n\}) \longrightarrow \nu_{H}^{\mathbf{1}_{\widehat{\Omega}}}(\{\rho\}).
\]
Since
\[
Z_{n,t}^{x_n}
:=
\bE^{x_n}\!\left[\exp\!\left(-\int_0^t V_n(X_s)\,\dd s\right)\right]
\]
is B-S continuous and by spectral decomposition,
\[
\varphi_n(\cdot)
=
\lim_{t\to\infty}
\frac{e^{-t\rho_n}Z_{n,t}(\cdot)}{\mu_n(\{\rho_n\})},
\]
with control of the convergence uniformly in $n$,
$$\| \varphi_n(\cdot) - \frac{e^{-t\rho_n}Z_{n,t}(\cdot)}{\mu_n(\{\rho_n\})}\|_{L^2(M_n)} \leq e^{- \epsilon t},$$
where $\epsilon$ is a lower bounds for the spectral gap. We obtain the converge of $\sqrt{\operatorname{Vol}(M_n)}\varphi_n(\cdot)$ to $\varphi(\cdot)$ in the Benjamini--Schramm sense. Finally, the $Q$-process on $M_n$ has transition density
\[
p_{L_n}(t,x,y)
= p_{{ \tfrac{\Delta_{M_n}}{2} }}(t,x,y)
e^{t\rho_n}\frac{\varphi_n(y)}{\varphi_n(x)}
\bE^{x}\!\left[
\left.\exp\!\left(-\int_0^t V_n(X_s)\,\dd s\right)\right| X_t=y
\right].
\]
Each terms in B-S concluding the proof.
\end{proof}

We end the paper with the periodization stability.

\begin{proof}[Proof of Theorem~\ref{thm:thm3}]
It is known that there exists a sequence of compact hyperbolic $d$-manifolds
$(M_n)$ that Benjamini--Schramm converges to $\Hd$ and has a uniform
spectral gap
\[
\lambda := \liminf_{n\to\infty}
\lambda_1\!\left(-\tfrac12\Delta_{M_n}\right) >0.
\]

This spectral gap follows from Clozel’s theorem \cite{clozel} that states that for any symmetric space \(X\) of a semisimple Lie group (in particular \( \Hd \)), there exists a constant \(\epsilon_d>0\), depending only on \(X\), such that every congruence lattice \(\Gamma\) in \(G\) yields an orbifold \(\Gamma\backslash X\) with spectral gap at least \(\epsilon_d\).

For cocompact arithmetic lattice \(\Gamma<\mathrm{Isom}(\Hd)\), by \cite{7sam}, any sequence of congruence subgroups of \(\Gamma\) Benjamini--Schramm converges to \(\Hd\). We define $\lambda_d$ to be the maximal value that can be achieved, and fix such a sequence $(M_n)$.

Let $\omega_n$ be a Poisson point process on
$M_n$. Since $V_\omega$ is a factor of the PPP, it can be approximated by
finite-range factors $V_{R_n,\omega}$, where $R_n\to\infty$ sufficiently slowly
compared to the injectivity radius at a uniformly chosen point of $M_n$.

Using the realization $\omega_n$ of the PPP on $M_n$, we define a periodic
potential $V_n$ on $M_n$ by copying the finite-range rule defining
$V_{R_n,\omega}$. By construction, the sequence $(M_n,V_n)$ satisfies the
assumptions of Theorem~\ref{thm:thm2}. Hence the associated $Q$-processes on
$M_n$ converge in the Benjamini--Schramm sense to the limiting $Q$-process on
$(\Hd,V_\omega)$.

Finally, lifting each $Q$-process on $M_n$ to the universal cover $\Hd$
produces a periodic $Q$-process. This yields the desired periodic approximation
of the limiting $Q$-process.
\end{proof}

\bibliographystyle{amsplain}   
\bibliography{refs}

\end{document}